\documentclass[a4paper]{amsart}
\usepackage{amssymb}

\input xy
\xyoption{all}

\newcommand{\pr}{\mathrm{pr}}
\newcommand{\Cl}{\mathrm{Cl}}

\newcommand{\C}{\mathcal C}

\newcommand{\B}{\mathcal B}

\newcommand{\R}{\mathbb R}
\newcommand{\N}{\mathbb N}
\newcommand{\F}{\mathcal F}

\newtheorem{theorem}{Theorem}

\newtheorem{lemma}{Lemma}

\newtheorem{proposition}{Proposition}
\newtheorem{ex}{Example}

\begin{document}

\title{Equilibrium under uncertainty with fuzzy payoff}
\author{Taras Radul}

\maketitle

Institute of Mathematics, Casimirus the Great University of Bydgoszcz, Poland;
\newline
Department of Mechanics and Mathematics, Ivan Franko National University of Lviv,
Universytettska st., 1. 79000 Lviv, Ukraine.
\newline
e-mail: tarasradul@yahoo.co.uk



\begin{abstract} This paper studies n-player games where players' beliefs about their opponents'
behaviour are capacities (fuzzy measures, non-additive probabilities). The concept of an equilibrium
under uncertainty was introduced by J.Dow and S.Werlang  (1994)   for two players and was extended  to n-player
games by  J.Eichberger and D.Kelsey  (2000).   Expected utility (payoff function) was expressed by Choquet
integral.  The concept of an equilibrium under uncertainty   with expected utility  expressed by Sugeno
integral were considered by T.Radul (2018). We consider in this paper an equilibrium with expected utility  expressed by fuzzy integral generated by a continuous t-norm which is a natural generalization of Sugeno
integral.
\end{abstract}

\maketitle

\section{Introduction}

The classical Nash equilibrium theory is based on fixed point theory and was developed in frames of linear convexity. The mixed strategies of a player are probability (additive) measures on a set of pure strategies. But an interest to Nash equilibria in more general frames is rapidly growing in last decades. For instance,
Aliprantis, Florenzano and Tourky \cite{AFT} work in ordered topological vector spaces, Luo \cite{L} in topological semilattices,
Vives \cite{Vi} in complete lattices.  Briec and Horvath \cite{Ch} proved  existence of Nash equilibrium point for $B$-convexity and MaxPlus convexity.

We can use additive measures only when we know precisely probabilities of all events considered in a game. However, it is not a case
in many modern economic models. The decision theory under uncertainty considers a model when probabilities of states are either not known or imprecisely specified. Gilboa \cite{Gil} and Schmeidler  \cite{Sch} axiomatized  expectations expressed by Choquet
integrals attached to non-additive measures called capacities (fuzzy measures), as a formal approach to decision-making under uncertainty. Dow and Werlang \cite{DW} used this approach for two players game where belief of each player about a choice of the strategy by the other player is a capacity. They introduced some equilibrium notion for such games and proved its existence.  This result was extended onto games with arbitrary finite number of players in \cite{EK}. Another interesting approach to the games in convex capacities with pay-off functions expressed by Choquet integrals can be find in \cite{Ma}.


An alternative to so-called Choquet expected utility model is the qualitative decision theory.   The corresponding expected utility is expressed by Sugeno integral. This approach was widely studied in the last decade  (\cite{DP},\cite{DP1},\cite{CH1},\cite{CH}).  Sugeno integral chooses a median value of utilities which is qualitative counterpart of the averaging operation by Choquet integral.

The equilibrium notion from \cite{DW} and \cite{EK} for a game  with expected payoff function defined by Sugeno integral was considered in \cite{R4}. The sets of pure strategies are arbitrarily compacta.  Let us remark that in \cite{DW} and \cite{EK} attention was restricted to convex capacities which play an important role in Choquet expected utility theory. There are two important classes of capacities in the qualitative decision theory, namely  possibility and necessity capacities which describe optimistic and pessimistic criteria \cite{DP}. The existence of equilibrium expressed by possibility (or necessity) capacities is proved in  \cite{R4}. Since the spaces of possibility and  necessity capacities have no natural linear convex structure, some non-linear convexity was  used.

Let us remark that the equilibrium notion from \cite{DW} supposed that players are allowed to form non-additive beliefs about opponent's decision and answer with pure strategies. Another approach was considered   \cite{KZ} and \cite{GM} for games with Choquet payoff where players are allowed to form non-additive beliefs about opponent's decision but also to play their mixed non-additive strategies expressed by capacities.  The same approach for games with Sugeno payoff was considered in \cite{R3}. Games with strategies expressed by possibility capacities were recently considered by Hosni and Marchioni \cite{HM}. They considered payoff functions represented by Choquet integral and Sugeno integral. Games with expected payoff functions represented by fuzzy integrals generated by the maximum operation and some continuous triangular norm (a partial case is the  Sugeno integral which is generated by the maximum and the minimum operations) were considered in \cite{R5}.

We consider the equilibrium notion from \cite{DW} with beliefs expressed by possibility (or necessity) capacities and payoff functions represented by fuzzy integrals in this paper. We prove existence of such equilibrium.

\section{Games with non-additive beliefs} By compactum we mean a compact Hausdorff space.  In what follows, all spaces are assumed to be compacta except for $\R$ and maps are assumed to be continuous. Let $A$ be a subset of $X$. By $\Cl A$ we denote the closure of $A$ in $X$. By $\F(X)$ we denote the family of all closed subsets of $X$.

We need the definition of capacity on a compactum $X$. We follow a terminology of \cite{NZ}.
A function $\nu:\F(X)\to [0,1]$  is called an {\it upper-semicontinuous capacity} on $X$ if the three following properties hold for each closed subsets $F$ and $G$ of $X$:

1. $\nu(X)=1$, $\nu(\emptyset)=0$,

2. if $F\subset G$, then $\nu(F)\le \nu(G)$,

3. if $\nu(F)<a$, then there exists an open set $O\supset F$ such that $\nu(B)<a$ for each compactum $B\subset O$.

If $F$ is a one-point set we use a simpler notation $\nu(a)$ instead $\nu(\{a\})$.
A capacity $\nu$ is extended in \cite{NZ} to all open subsets $U\subset X$ by the formula $\nu(U)=\sup\{\nu(K)\mid K$ is a closed subset of $X$ such that $K\subset U\}$.

It was proved in \cite{NZ} that the space $MX$ of all upper-semicontinuous  capacities on a compactum $X$ is a compactum as well, if a topology on $MX$ is defined by a subbase that consists of all sets of the form $O_-(F,a)=\{c\in MX\mid c(F)<a\}$, where $F$ is a closed subset of $X$, $a\in [0,1]$, and $O_+(U,a)=\{c\in MX\mid c(U)>a\}$, where $U$ is an open subset of $X$, $a\in [0,1]$. Since all capacities we consider here are upper-semicontinuous, in the following we call elements of $MX$ simply capacities.

A capacity $\nu\in MX$ for a compactum $X$ is called  a necessity (possibility) capacity if for each family $\{A_t\}_{t\in T}$ of closed subsets of $X$ (such that $\bigcup_{t\in T}A_t$ is a closed subset of $X$) we have $\nu(\bigcap_{t\in T}A_t)=\inf_{t\in T}\nu(A_t)$  ($\nu(\bigcup_{t\in T}A_t)=\sup_{t\in T}\nu(A_t)$). (See \cite{WK} for more details.) We denote by $M_\cap X$ ($M_\cup X$) a subspace of $MX$ consisting of all necessity (possibility) capacities. Since $X$ is compact and $\nu$ is upper-semicontinuous, $\nu\in M_\cap X$ iff $\nu$ satisfy the simpler requirement that $\nu(A\cap B)=\min\{\nu(A),\nu(B)\}$.

If $\nu$ is a capacity on a compactum $X$, then  the function $\kappa X(\nu)$, that is defined on the family $\F(X)$  by the formula $\kappa X(\nu)(F) = 1-\nu (X\setminus F)$, is a capacity as well. It is called the dual
capacity (or conjugate capacity ) to $\nu$. The mapping $\kappa X : MX \to MX$ is a homeomorphism and an involution \cite{NZ}. Moreover, $\nu$ is a necessity capacity if and only if $\kappa X(\nu)$ is a possibility capacity. This implies in particular that $\nu\in M_\cup X$ iff $\nu$ satisfy the simpler requirement that $\nu(A\cup B)=\max\{\nu(A),\nu(B)\}$. It is easy to check that $M_\cap X$ and $M_\cup X$ are closed subsets of $MX$.

For each capacity $\nu$ we consider an upper semicontinuous function $[\nu]:X\to [0,1]$ that sends each $x\in X$ to $\nu(x)$.
Observe that for a possibility capacity $\nu\in M_\cup X$ and a closed set $F\subset X$ we have $\nu(F) =
\max\{\nu(x) | x \in F\}$, and $\nu$ is completely determined by its values on singletons. It means that $\nu$ is completely determined by the function $[\nu]$.
 Conversely, each
upper semicontinuous function $f:X\to I$ with $\max f = 1$ determines a possibility capacity $(f)\in M_\cup X$
by the formula $(f)(F) = \max\{f(x) | x \in F\}$, for a closed subset $F$ of $X$.

 For a continuous map of compacta $f:X\to Y$ we define the map $Mf:MX\to MY$ by the formula $Mf(\nu)(A)=\nu(f^{-1}(A))$ where $\nu\in MX$ and $A\in\F(Y)$. The map $Mf$ is continuous.  In fact, this extension of the construction $M$ defines the capacity functor in the category of compacta and continuous maps. The categorical technics are very useful for investigation of capacities on compacta (see \cite{NZ} for more details). We try to avoid the formalism of category theory in this paper,  but we follow the main ideas of such approach.

\section{Tensor products of capacities} For a continuous map of compacta $f:X\to Y$ we define the map $Mf:MX\to MY$ by the formula $Mf(\nu)(A)=\nu(f^{-1}(A))$ where $\nu\in MX$ and $A\in\F(Y)$. The map $Mf$ is continuous.  In fact, this extension of the construction $M$ defines the capacity functor in the category of compacta and continuous maps. The categorical technics are very useful for investigation of capacities on compacta (see \cite{NZ} for more details). We try to avoid the formalism of category theory in this paper,  but we follow the main ideas of such approach.

The tensor product operation of probability measures is well known and very useful partially for investigation of the spaces of probability measures on compacta (see for example Chapter 8 from \cite{FF}). General categorical definition of tensor product for any functor was given in \cite{BR}. Applying this definition to the capacity functor we obtain that a tensor product of capacities on compacta $X_1$ and $X_2$ is  a continuous map $$\otimes:MX_1\times MX_2\to M(X_1\times X_2)$$ such that for each $i\in\{1,2\}$ we have $M(p_i)\circ\otimes= \pr_i$ where $p_i:X_1\times X_2\to X_i$, $\pr_i:MX_1\times MX_2\to MX_i$ are the corresponding projections.

A tensor product for capacities was introduced in \cite{KZ}.  This definition is based on the capacity monad structure.   An explicit formula for evaluating  tensor product of capacities was given in \cite{R4} omitting the formalism of category theory. For $\mu_1\in MX_1$, $\mu_2\in MX_2$ and $B\in\F(X_1\times X_2)$ we put $$\mu_1\otimes\mu_2(B)=\max\{t\in[0,1]\mid\mu_1(\{x\in X_1\mid \mu_2(p_2((\{x\}\times X_2)\cap B))\ge t\})\ge t\}.$$

The problem of multiplication of capacities was deeply considered in the possibility theory and its application to the game theory and the decision making theory where the term joint possibility distribution is used. A standard choice  of a joint possibility distribution is based on the minimum operation (see for example \cite{HM}). For $\mu_1\in M_\cup X_1$, $\mu_2\in M_\cup X_2$ and $(x,y)\in X_1\times X_2$ we put $$[\mu_1\otimes\mu_2](x,y)=[\mu_1](x)\wedge[\mu_2](y).$$ (Let us remind that by $[\nu]$ we denote the density of a possibility capacity $\nu$.) The coincidence of both definitions is proved in \cite{R5}. So, the difference is only in terms. It is also shown in \cite{R5} that  the notion of tensor product coincides with the method of aggregation of capacities considered in \cite{Df}.

 A more general approach is also used where the minimum operation is changed by any t-norm (see for example \cite{DP2}). We will use this definition in our paper but we  prefer the term tensor product.

Remind that triangular norm $\ast$ is a binary operation on the closed unit interval $[0,1]$ which is associative, commutative, monotone and $s\ast 1=s$ for each  $s\in [0,1]$ \cite{PRP}. We consider only continuous t-norms in this paper.

So, we fix a continuous t-norm $\ast$ and consider a tensor product $\circledast:M_\cup X_1\times M_\cup X_2\to M_\cup(X_1\times X_2)$ generated by $\ast$ defined as follows. For possibility capacities $\mu_1\in M_\cup X_1$, $\mu_2\in M_\cup X_2$ and $(x,y)\in X_1\times X_2$ we put $$[\mu_1\circledast\mu_2](x,y)=[\mu_1](x)\ast[\mu_2](y).$$

We also can generalize the above mentioned formula from \cite{R4}. For capacities $\mu_1\in MX_1$, $\mu_2\in MX_2$ and $B\in\F(X_1\times X_2)$ we put $$\mu_1\widetilde{\circledast}\mu_2(B)=\sup\{t\in[0,1]\mid\mu_1(\{x\in X_1\mid \mu_2(p_2((\{x\}\times X_2)\cap B))\ge t\})\ast t\}.$$ It was shown in \cite{R5}  that  both definitions coincide in the class of possibility capacities.

It was noticed in \cite{KZ} that we can extend the definition of tensor product to any finite number of factors by induction. It is also true for the tensor product generated by a continuous norm.

\begin{lemma}\label{CP} The map $\widetilde{\circledast}:M(X_1)\times\dots\times M(X_n)\to M(X_1\times\dots\times X_n)$ is continuous.
\end{lemma}

\begin{lemma}\label{P} Let $X_i$ be a compactum, $A_i\in\F(X_i)$ and $\mu_i\in MX_i$ such that $\mu_i(X_i\setminus A_i)=0$ for each $i\in \{1,\dots,n\}$. Then $\widetilde{\circledast}_{i=1}^n\mu_i(\prod_{i=1}^n X_i\setminus\prod_{i=1}^n A_i)=0$.
\end{lemma}

\begin{proof} Consider the case $n=2$.   Let $B$ any compact subset of $(X_1\times X_2)\setminus(A_1\times A_2)$. For any $t>0$ consider the set $K_t=\{x\in X_1\mid \mu_2(p_2((\{x\}\times X_2)\cap B))\ge t\}$. If $x\in A_1$, then $p_2((\{x\}\times X_2)\cap B)\subset X_2\setminus A_2$, hence $x\notin K_t$. Thus $K_t\subset X\setminus A_1$ and we obtain $\mu_1\widetilde{\circledast}\mu_2(B)=0$.

The general case could be obtained by induction.
\end{proof}

\section{Equilibrium under uncertainty with fuzzy payoff}\label{equil}

Let us describe the fuzzy integral generated by a continuous t-norm $\ast$ with respect to a capacity $\mu\in MX$. Such integrals  are called t-normed integrals and were studied in \cite{We1}, \cite{We2}, \cite{Sua} and \cite{R5}. Denote $\varphi_t=\varphi^{-1}([t,+\infty))$ for each $\varphi\in C(X,[0,1])$ and $t\in[0,1]$.
For a continuous t-norm $\ast$ and a  function $f\in  C(X,[0,1])$ the corresponding t-normed integral is defined by the formula
$$\int_X^{\vee\ast} fd\mu=\max\{\mu(f_t)\ast t\mid t\in[0,1]\}.$$
Let us remark that  existence of maximum in the previous definition follows from the semicontinuity of the capacity $\mu$. If we consider a partial case when t-norm is the minimum operation, we obtain the Sugeno integral.

 Now, we are going to introduce notion of  equilibrium under uncertainty for games where belief of each player about a choice of the strategy by the other player is a capacity. We follow definitions and denotation from \cite{EK} with the only difference that we use the t-normed integral for expected payoff instead  the Choquet integral. Our approach is a generalization of \cite{R4} where expected payoff was expressed by the Sugeno integral.

 We consider a $n$-players game $p:X=\prod_{i=1}^n X_i\to[0,1]^n$ with compact Hausdorff spaces of strategies $X_i$. We assume that the function $p$ is continuous. The coordinate function $p_i:X\to [0,1]$ we call payoff function of $i$-th player. For $i\in\{1,\dots,n\}$ we denote by  $X_{-i}=\prod_{j\neq i} X_j$  the set of strategy combinations which players other
than $i$ could choose. For $x\in X$ the corresponding point in $X_{-i}$ we denote by $x_{-i}$.  In contrast to standard game theory, beliefs of $i$-th player about opponents'
behaviour are represented by non-additive measures (or capacities) on $X_{-i}$.

Let $\ast$ be a continuous t-norm.
For $i\in\{1,\dots,n\}$  we consider the expected payoff function $P^\ast_i:X_i\times MX_{-i}\to[0,1]$ defined as follows $P^\ast_i(x_i,\nu)=\int_{X_{-i}}^{\vee\ast} p_{i}^{x_i}d\nu$ where the function $p_{i}^{x_i}:X_{-i}\to[0,1]$ is defined by the formula $p_{i}^{x_i}(x_{-i})=p_{i}(x_i,x_{-i})$, $x_i\in X_i$ and $\nu\in MX_{-i}$.

We are going to prove continuity of $P^\ast_i$. We will need some notations and a  technical lemma from \cite{R4}. Let $f:X\times Y\to[0,1]$ be a  function. Consider any $x\in X$ and $t\in[0,1]$. Denote $f^x_{\le t}=\{y\in Y\mid f(x,y)\le t\}$. We also will use analogous notations $f^x_{\ge t}$, $f^x_{<t}$ and $f^x_{>t}$. Let us remark that we already used the shorter notation $f_t=f^x_{\ge t}$ for the case when $X=\{x\}$.

\begin{lemma}\cite{R4}\label{C0} Let $f:X\times Y\to[0,1]$ be a continuous function on the product $X\times Y$ of compacta $X$ and $Y$. Then for each $x\in X$, $t\in[0,1]$ and $\delta>0$ there exists an open neighborhood $O$ of $x$ such that $f^z_{\le t}\subset f^x_{<t+\delta}$ ($f^z_{\ge t}\subset f^x_{>t-\delta}$) for each $z\in O$.
\end{lemma}

\begin{lemma}\label{C} The map $P^\ast_i$ is continuous.
\end{lemma}

\begin{proof} Consider any $x\in X_i$ and $\nu_0\in MX_{-i}$ and put $P^\ast_i(x,\nu_0)=s\in[0,1]$. Denote $f=p_{i}$.
Consider any $\delta>0$. Since t-norm $\ast$ is uniformly  continuous, we can choose $\varepsilon>0$ such that $|r\ast l-p\ast t|<\delta/4$ for each $r$, $l$, $p$, $t\in[0,1]$ such that $|r-p|<\varepsilon$ and $|l- t|<\varepsilon$. Choose $n\in\N$ such that $1/n\le\varepsilon$ and put $t_i=i/n$ for $i\in\{0,\dots,n\}$. By Lemma \ref{C0} applied to the continuous function $f:X_i\times X_{-i}\to[0,1]$ we can choose a neighborhood $O_1$ of $x$ such that for each $z\in O_1$ we have $f^z_{\ge t_i}\subset f^x_{\ge t_{i-1}}$ for each $i\in\{1,\dots,n\}$. Put $V_1=\{\nu\in M(X_{-i})\mid \nu(f^x_{\ge t_i})<\nu_0(f^x_{\ge t_i})+\varepsilon\}$ for each $i\in\{0,\dots,n\}$, then $V_1$ is a neighborhood of $\nu_0$.

Consider any  $(z,\nu)\in O_1\times V_1$ and $t\in[0,1]$. Choose $i\in\{0,\dots,n-1\}$ such that $t\in[t_i,t_{i+1}]$. Since $t_i\le t$ and $|t_i-t|<\varepsilon$,  we have $\nu(f^z_{\ge t})\ast t\le\nu(f^z_{\ge t_i})\ast t_i+\delta/4$.

If $i=0$, we have  $\nu(f^z_{\ge t})\ast t<\delta/4\le s+\delta/4$.

Consider the case when $i>0$. Since $z\in O_1$ and $\ast$ is monotone, we have $\nu(f^z_{\ge t_i})\ast t_i+\delta/4\le\nu(f^x_{\ge t_{i-1}})\ast t_{i-1}+\delta/4$.

If $\nu(f^x_{\ge t_{i-1}})\le\nu_0(f^x_{\ge t_{i-1}})$, we have $\nu(f^x_{\ge t_{i-1}})\ast t_{i-1}+\delta/4\le\nu_0(f^x_{\ge t_{i-1}})\ast t_{i-1}+\delta/4\le s+\delta/4$.

If $\nu(f^x_{\ge t_{i-1}})>\nu_0(f^x_{\ge t_{i-1}})$, we have $|\nu(f^x_{\ge t_{i-1}})-\nu_0(f^x_{\ge t_{i-1}})|$, because $\nu\in V_1$. Then we obtain $\nu(f^x_{\ge t_{i-1}})\ast t_{i-1}+\delta/4\le\nu_0(f^x_{\ge t_{i-1}})\ast t_{i-1}+\delta/2\le s+\delta/2$.

Therefore we have $\nu(f^z_{\ge t})\ast t\le s+\delta/2$ for each $t\in[0,1]$. Hence $P_i(z,\nu)<s+\delta$ for each  $(z,\nu)\in O_1\times V_1$.

Choose $d\in[0,1]$ such that $\nu_0(f^x_{\le d})\ast d=\max\{\mu(f^x_{\le t})\ast t\mid t\in[0,1]\}=s$.

By Lemma \ref{C0} there exists a neighborhood $O_2$ of $x$ such that for each $z\in O_2$ we have $f^z_{\le d-\frac\varepsilon2}\subset f^x_{<d-\frac\varepsilon4}$. Put $V_2=\{\nu\in M(X_{-i})\mid \nu(f^x_{\ge d})>\nu_0(f^x_{\ge d})-\varepsilon\}$, then $V_2$ is a neighborhood of $\nu_0$.

Then for each $(z,\nu)\in O_2\times V_2$ we have $\nu(f^z_{\ge d-\frac\varepsilon2})\ge\nu(f^z_{>d-\frac\varepsilon2})=\nu(X_{-i}\setminus f^z_{\le d-\frac\varepsilon2})\ge\nu(X_{-i}\setminus f^x_{<d-\frac\varepsilon4})=\nu(f^x_{\ge d-\frac\varepsilon4})\ge\nu(f^x_{\ge d})>\nu_0(f^x_{\ge d})-\varepsilon$.

If $\nu(f^z_{\ge d-\frac\varepsilon2})\ge\nu_0(f^x_{\ge d})$, then $\nu(f^z_{\ge d-\frac\varepsilon2})\ast(d-\frac\varepsilon2)\ge\nu_0(f^x_{\ge d})\ast(d-\frac\varepsilon2)>s-\delta/4$.

Consider the case $\nu(f^z_{\ge d-\frac\varepsilon2})<\nu_0(f^x_{\ge d})$. We also have $\nu(f^z_{\ge d-\frac\varepsilon2})>\nu_0(f^x_{\ge d})-\varepsilon$. Hence $\nu(f^z_{\ge d-\frac\varepsilon2})\ast(d-\frac\varepsilon2)>s-\delta/4$. So, we obtain that  $P^\ast_i(z,\nu)>s-\delta$ for each $(z,\nu)\in O_2\times V_2$ and the map $P^\ast_i$ is continuous.

\end{proof}

For $\nu_i\in M(X_{-i})$ denote by $$R^\ast_i(\nu_i)=\{x\in X_i\mid P^\ast_i(x,\nu_i)=\max\{P^\ast_i(z,\nu_i)\mid z\in X_i\}\}$$ the best response correspondence
of player $i$ given belief $\nu_i$. The set $R_i(\nu_i)$ is well defined and compact by Lemma \ref{C}.

A belief system $(\nu_1,\dots,\nu_n)$, where $\nu_i\in M(X_{-i})$, is called {\it an equilibrium under uncertainty respect $\ast$ with fuzzy payoff} if for all $i$ we have $\nu_i(X_{-i}\setminus\prod_{j\ne i}R^\ast_j(\nu_j))=0$.

    The main goal of this paper is to prove the existence of such equilibrium where corresponding belief system consist of possibility measures. Since the space $M_\cup X$ has no natural linear convex structure, we will use some another natural convexity structure on the space of capacities described in the next section.

\section{A convexity on the space of capacities}

Consider a compactum $X$. There exists a natural lattice structure on $MX$ defined as follows $\nu\vee\mu(A)=\max\{\nu(A),\mu(A)\}$ and $\nu\wedge\mu(A)=\min\{\nu(A),\mu(A)\}$ for each closed subset $A\subset X$ and $\nu$, $\mu\in MX$ (see for instance Theorem 7.1 in \cite{Ob}).The lattice $MX$ has a greatest  element and a a least element defined as $\mu_{1X}(A)=1$ for each $A\neq \emptyset$, $\mu_{1X}(\emptyset)=0$ and $\mu_{0X}(A)=0$ for each $A\neq X$, $\mu_{0X}(X)=1$.

Following \cite{W} we call a family $\C\subset\F(X)$ a convexity on a compactum $X$ if $\C$ is stable for intersection
and contains $X$ and the empty set. Elements of $\C$ are called
$\C$-convex (or simply convex). Let us remark that this definition  is different from the definition of convexity in \cite{vV} where the abstract  convexity theory is covered from the axioms to applications in different areas. We consider here only closed convex sets and we do not consider the condition that union of each nested subfamily of $\C$ is in $\C$. In fact considered in this paper convexities have such property, but we do not need it for our purposes. So, we use the simpler definition from \cite{W}.

A convexity $\C$ on $X$ is called $T_4$ (normal) if for each disjoint $C_1$, $C_2\in
\C$ there exist $S_1$, $S_2\in\C$ such that $S_1\cup S_2=X$,
$C_1\cap S_2=\emptyset$ and $C_2\cap S_1=\emptyset$.

The main goal of this section is to choose an appropriate normal convexity on  the compactum $M_\cup X$. There is no natural linear structure on $M_\cup X$, so, we can not use classical linear convexity. The idempotent max-plus convexity on the space of homogeneous possibility (or necessity) capacities was considered in \cite{FGM}. Since we deal with fuzzy integrals defined by the maximum operation and a t-norm $\ast$, it seems to be natural to consider the idempotent  max-$\ast$ convexity. But we do not know if such convexity is normal. So,  we consider here a coarser convexity on $M_\cup X$ used in \cite{R4}   which follows from a general categorical approach developed in \cite{R1} and \cite{R3}.

For $\nu$, $\mu\in MX$ we denote $[\nu,\mu]=\{\alpha\in MX\mid \nu\wedge\mu\le\alpha\le\nu\vee\mu\}$. It is easy to see that $[\nu,\mu]$ is a closed subset of $MX$. We consider   on the compactum $M_\cup X$ the  convexity  $\C_{\cup X}=\{C\cap M_\cup X\mid C\in \C_X\}$. It is easy to see that a closed subset $A$ of $M_\cup X$ belongs to the family  $\C_{\cup X}$ if and only if $[\bigwedge A,\bigvee A]\cap M_\cup X=A$. Let us remark that  $\bigvee A\in M_\cup X$ but we can not state it about  $\bigwedge A$.

The following lemma was proved in \cite{R4}.

\begin{lemma}\cite{R4}\label{norm} The convexity $\C_X$ is normal.
\end{lemma}

\section{The main result}

By a multimap (set-valued map) of a set $X$ into a set $Y$ we mean a map $F:X\to 2^Y$. We use the notation $F:X\multimap Y$. If $X$ and $Y$ are topological spaces, then a multimap $F:X\multimap Y$ is called upper semi-continuous (USC) provided for each open set $O\subset Y$ the set $\{x\in X\mid F(x)\subset O\}$ is open in $X$. It is well-known that a multimap between compacta $X$ and $Y$ with closed values  is USC iff its graph is closed in $X\times Y$.

Let  $F:X\multimap X$ be a  multimap. We say that a point $x\in X$ is a fixed point of $F$ if $x\in F(x)$.
The following counterpart of Kakutani theorem for abstract convexity is a partial case of Theorem 3 from \cite{W}.

\begin{theorem}\cite{W}\label{KA} Let $\C$ be a normal convexity on a compactum $X$ such that all convex sets are connected and $F:X\multimap X$ is a USC multimap with values in $\C\setminus\{\emptyset\}$. Then $F$ has a fixed point.
\end{theorem}

The convexity $\C_{\cup X}$ is normal. We need also connectedness to apply the previous theorem. A stronger statement was proved in \cite{R4}.

\begin{lemma}\cite{R4}\label{connected} Each element of the convexity $\C_{\cup X}$ is path connected.
\end{lemma}

We use definitions and notations from Section \ref{equil}. The following theorem is a generalization of Theorem 3 from \cite{R4} and arguments in the proofs are the same. But, for sake of completeness we give here a complete proof.

\begin{theorem} Let  $\star$ and $\ast$ be two continuous t-norms. There exists $(\mu_1,\dots,\mu_n)\in M_\cup(X_1)\times\dots\times M_\cup(X_n)$ such that $(\mu_1^*,\dots,$ $\mu_n^*)$ is an   equilibrium under uncertainty respect $\star$ with fuzzy payoff, where $\mu_i^*=\circledast_{j\neq i}\mu_j$.
\end{theorem}

\begin{proof}  For each $i\in\{1,\dots,n\}$ consider a multimap $\gamma_i:\prod_{j=1}^nM_\cup(X_j)\multimap M_\cup(X_i)$ defined as follows $\gamma_i(\mu_1,\dots\mu_n)=\{\mu\in M_\cup(X_i)\mid \mu(X_i\setminus R^\star_i(\mu_i^*))=0\}$. It follows from the definition of topology on $M_\cup(X_i)$ that $\gamma_i(\mu_1,\dots\mu_n)$ is a closed subset of $M_\cup(X_i)$ for each  $(\mu_1,\dots\mu_n)\in \prod_{j=1}^nM_\cup(X_j)$. Consider $\nu\in M_\cup(X_i)$ defined as follows $\nu(A)=1$ if $A\cap R_i(\mu_i^*)\neq\emptyset$ and $\nu(A)=0$ otherwise.
Then we have  $\gamma_i(\mu_1,\dots\mu_n)=[\mu_{X_i0},\nu]\cap M_\cup(X_i)$, hence $\gamma_i(\mu_1,\dots\mu_n)\in\C_{\cup X_i}$.

Define a multimap $\gamma:\prod_{j=1}^nM_\cup(X_j)\multimap \prod_{j=1}^nM_\cup(X_j)$ by the formula $\gamma(\mu_1,\dots,$ $\mu_n)=\prod_{i=1}^n\gamma_i(\mu_1,\dots,\mu_n)$. Let us show that $\gamma$ is USC. Consider any pair $(\mu,\nu)\in\prod_{j=1}^nM_\cup(X_j)\times\prod_{j=1}^nM_\cup(X_j)$ such that $\nu\notin \gamma(\mu)$. Then there exists $i\in\{1,\dots,n\}$ and a compactum $K\subset X_i\setminus R^\star_i(\mu_i^*)$ such that $\nu_i(K)>0$. Put $O_\nu\{\alpha\in \prod_{j=1}^nM(X_j)\mid \alpha_i(K)>0\}$. Then $O_\nu$ is an open neighborhood of $\nu$. It follows from Lemma  \ref{C} and continuity of tensor product that there exists an open neighborhood $O_\mu$ of $\mu$ such that for each $\alpha\in O_\mu$ we have $R^\star_i(\alpha_i^*)\cap K=\emptyset$. Hence for each $(\alpha,\beta)\in O_\mu\times O_\nu$ we have $\beta\notin\gamma(\alpha)$ and $\gamma$ is USC.

We consider on $\prod_{j=1}^nM_\cup(X_j)$ the family $\C=\{\prod_{i=1}^n C_i\mid C_i\in\C_{\cup X_i}\}$. It is easy to see that $\C$ forms a normal convexity  on a compactum $\prod_{j=1}^nM_\cup(X_j)$ such that all convex sets are connected. Then by Theorem \ref{KA} $\gamma$ has a fixed point $\mu=(\mu_1,\dots\mu_n)\in \prod_{j=1}^nM_\cup(X_j)$. Let us show that $(\mu_1^*,\dots,\mu_n^*)$ is an   equilibrium under uncertainty respect $\star$. Consider any $i\in\{1,\dots,n\}$. Then $\mu_i(X_i\setminus R^\star_i(\mu_i^*))=0$. We have by Lemma \ref{P} $\mu_i^*(\prod_{j\ne i}X_i\setminus\prod_{j\ne i}R^\star_j(\mu_j^*))=0$.
\end{proof}

We can define on $M_\cap X$ a convexity $\C_{\cap X}=\{C\cap M_\cap X\mid C\in \C_X\}$. The homeomorphism $\kappa X:M_\cup X\to M_\cap X$ is an isomorphism of convex structures $\C_{\cup X}$ and $\C_{\cap X}$ (more precisely $C\in\C_{\cup X}$ iff $\kappa X(C)\in\C_{\cap X}$). Hence, using the same arguments as before, we obtain the following theorem.

\begin{theorem} Let  $\star$ and $\ast$ be two continuous t-norms. There exists $(\mu_1,\dots,\mu_n)\in M_\cap(X_1)\times\dots\times M_\cap(X_n)$ such that $(\mu_1^*,\dots,$ $\mu_n^*)$ is an   equilibrium under uncertainty with fuzzy payoff respect $\star$, where $\mu_i^*=\circledast_{j\neq i}\mu_j$.
\end{theorem}


\section{Some examples} As we remark before, our result is a generalisation of results obtained in \cite{R5}.  In this section we consider some examples which demonstrate that this generalization is not trivial and distinguish our approach between approaches considered before.

\begin{ex}
Consider the 2-person game in pure strategies $u:\{a,b\}\times\{a,b\}\to\R^2$, where  $u_1(a,a)=1/2$, $u_1(a,b)=0$, $u_1(b,a)=0$, $u_1(b,b)=1/2$ and $u_2(a,a)=0$, $u_2(a,b)=1/2$, $u_2(b,a)=1/2$, $u_2(b,b)=0$. Choose $\nu\in M_\cup (\{a,b\})$ defined by its density $[\nu](a)=1$ and $[\nu](b)=1/2$.  Let us show that the pair $(\nu,\nu)$ is  an   equilibrium under uncertainty respect t-norm $\wedge$ (by $\wedge$ we denote the minimum operation on $[0,1]$.

We have $P_1^\wedge(a,\nu)=\max\{\nu(f_t)\wedge t\mid t\in[0,1]\}=1\wedge1/2=1/2$ where $f(a)=1/2$ and $f(b)=0$. We also have $P_1^\wedge(b,\nu)=\max\{\nu(g_t)\wedge t\mid t\in[0,1]\}=1/2\wedge1/2=1/2$ where $g(a)=0$ and $g(b)=1/2$. Hence $R_1^\wedge(\nu)=\{a,b\}$. Analogously we can check that $R_2^\wedge(\nu)=\{a,b\}$. Hence $(\nu,\nu)$ is  an equilibrium under uncertainty respect $\wedge$.

 Now, consider another well-known continuous t-norm $\cdot$ (multiplication). Then  we have $P_1^\cdot(a,\nu)=\max\{\nu(f_t)\cdot t\mid t\in[0,1]\}=1\cdot1/2=1/2$ and $P_1^\cdot(b,\nu)=\max\{\nu(g_t)\cdot t\mid t\in[0,1]\}=1/2\cdot1/2=1/4$. Hence $R_1^\cdot(\nu)=\{a\}$ and $(\nu,\nu)$ is  not an equilibrium under uncertainty respect $\cdot$.

This example demonstrate that  equilibrium depends on t-norm, so, our generalization is not trivial.
\end{ex}

Another approach to  equilibrium is considered in  \cite{R5} where players are allowed  to play their mixed non-additive strategies expressed by capacities and expected payoff functions are represented by fuzzy integral. Let us describe this construction in  more detail.

We consider a game $u:Z=\prod_{i=1}^n Z_i\to[0,1]^n$ with compact Hausdorff spaces of pure strategies $Z_1,\dots,Z_n$ and continuous payoff functions $u_i:\prod_{i=1}^n Z_i\to[0,1]$. Let  $\star$ and $\ast$ be two t-norms.  We will extend the game $u:Z=\prod_{i=1}^n Z_i\to[0,1]^n$ to a game in mixed strategies  $eu:\prod_{i=1}^n M_\cup Z_i\to[0,1]^n$ using the integral generated by t-norm $\star$ and the tensor product generated by t-norm $\ast$.

 We define expected payoff functions $eu_i:\prod_{j=1}^n M_\cup Z_j\to[0,1]$ by the formula  $$eu_i(\nu_1,\dots,\nu_n)=\int_X^{\vee\star} u_i d(\nu_1\circledast\dots\circledast\nu_n)$$
 for $(\nu_1,\dots,\nu_n)\in\prod_{j=1}^n M_\cup Z_j$. We are looking for Nash equilibrium for such game.

 There exists a trivial solution of the problem of existence of Nash equilibrium. We can consider the natural order on $M_\cup Z_i$. Then each $M_\cup Z_i$ contains the greatest element $\mu_i$ defined by the formula $$\mu_i(A)=\begin{cases}
0,&A=\emptyset,\\
1,&A\neq\emptyset\end{cases}$$ for $A\in\F(Z_i)$. It is easy to see that $(\mu_1,\dots,\mu_n)$ is a Nash equilibrium point.  (In fact, the main attention in \cite{R5} is paid to the model, when the goal of each player is to minimize his expected payoff function. Since $M_\cup Z_i$ does not contain the smallest element,  existence of Nash equilibrium is not trivial for such games.) The following example shows that it is not the case for the equilibrium  under uncertainty.

\begin{ex}
Consider the 2-person game in pure strategies $u:\{a,b\}\times\{a,b\}\to\R^2$, where  $u_1(a,a)=1$, $u_1(a,b)=0$, $u_1(b,a)=1/2$, $u_1(b,b)=1/2$ and $u_2(a,a)=0$, $u_2(a,b)=1$, $u_2(b,a)=1/2$, $u_2(b,b)=1/2$. Then the capacity $\nu\in M_\cup (\{a,b\})$ defined by its density $[\nu](a)=1$ and $[\nu](b)=1$ is the greatest element of $M_\cup (\{a,b\})$.  Evidently the pair $(\nu,\nu)$ is  an Nash  equilibrium for the game with expected payoff functions $eu_i$ (we take $\wedge$ as both t-norms in the defining of payoff).

On the other hand we have $P_1^\wedge(a,\nu)=\max\{\nu(f_t)\wedge t\mid t\in[0,1]\}=1\wedge1=1$ where $f(a)=1$ and $f(b)=0$. We also have $P_1^\wedge(b,\nu)=\max\{\nu(g_t)\wedge t\mid t\in[0,1]\}=1/2\wedge1=1/2$ where $g(a)=1/2$ and $g(b)=1/2$. Hence $R_1^\wedge(\nu)=\{a\}$.  Hence $(\nu,\nu)$ is not  an equilibrium under uncertainty respect $\wedge$.
\end{ex}

However we have the following proposition.

\begin{proposition} Let  $(f,g):X\times Y\to [0,1]^2$ be a game in pure strategies with compact $X$ and $Y$ and continuous payoff functions $f$ and $g$. If a pair of capacities $(\nu,\mu)\in M_\cup(X)\times M_\cup(Y)$ is an equilibrium under uncertainty respect a t-norm $\star$, then $(\nu,\mu)$ is a point of Nash equilibrium  of the corresponding game in capacities with payoff functions $ef$ and $eg$ defined by fuzzy integral respect  $\star$ and tensor product generated by any  t-norm $\ast$.
\end{proposition}

\begin{proof} Suppose the contrary $(\nu,\mu)$ is not a point of Nash equilibrium. We can assume, without loss of generality, that there exists $\nu'\in M_\cup(X)$ such that $ef(\nu',\mu)>ef(\nu,\mu)$. Let $\nu_1$ be the maximal element of $M_\cup(X)$ (it means that $\nu_1(A)=1$  for each closed non-empty subset $A$ of $X$ or equivalently  $[\nu_1](x)=1$ for each $x\in X$). It is easy to see that $ef(\nu',\mu)\le f(\nu_1,\mu)$, hence  we have $ef(\nu_1,\mu)>ef(\nu,\mu)$. We also have $ef(\nu_1,\mu)=\max\{\nu_1\circledast\mu(f_t)\star t\mid t\in[0,1]\}=\max\{\max\{[\nu_1](x)\ast[\mu](y)\mid(x,y)\in f_t\}\star t\mid t\in[0,1]\}=\max\{\max\{[\mu](y)\mid(x,y)\in f_t\}\star t\mid t\in[0,1]\}$. Hence there exist $t_0\in[0,1]$ and $(x_0,y_0)\in f_{t_0}$ such that $[\mu](y_0)\star t_0>ef(\nu,\mu)$. Since $\nu\in M_\cup(X)$, there exists $x_1\in X$ such that $[\nu](x)=1$. Then we have $ef(\nu,\mu)=\max\{\max\{[\nu](x)\ast[\mu](y)\mid(x,y)\in f_t\}\star t\mid t\in[0,1]\}\ge\max\{\max\{[\mu](y)\mid(x_1,y)\in f_t\}\star t\mid t\in[0,1]\}=P_1^\star(x_1,\mu)$ and $P_1^\star(x_0,\mu)\ge[\mu](y_0)\star t_0>ef(\nu,\mu)\ge P_1^\star(x_1,\mu)$. Hence $x_1\notin R_1^\star(\mu)$. Since $\nu(\{x_1\})=1$, we obtain a contradiction.

\end{proof}


\begin{thebibliography}{}



\bibitem{AFT} C.D.Aliprantis, M.Florencano, R.Tourky {\em General equilibrium analisis in ordered topological vector spaces,} J. Math. Econom. 40 (2004)
247--269.

\bibitem{BR} T.Banakh,  T.Radul {\em F-Dugundji spaces, F-Milutin spaces and absolute F-valued retracts,} Topology Appl. {\bf 179} (2015), 34--50.

\bibitem{Ch}  W.Briec, Ch.Horvath {\em Nash points, Ku Fan inequality and equilibria of abstract economies in Max-Plus and $\B$-convexity,} J. Math. Anal. Appl. {\bf 341} (2008), 188--199.

\bibitem{Ob}  G.L.O'Brien, W.Verwaat  {\em How subadditive are subadditive capacities?,} Comment. Math. Univ.
Carolinae {\bf 35} (1994), 311--324.

\bibitem{CH1} A. Chateauneuf, M. Grabisch, A. Rico,   {\em Modeling attitudes toward uncertainty through the use
of the Sugeno integral}, Journal of Mathematical Economics {\bf 44} (2008) 1084--1099.


\bibitem{DW} J.Dow, S.Werlang, {\em Nash equilibrium under Knightian uncertainty: breaking down backward induction},J Econ. Theory {\bf 64} (1994) 205--224.

\bibitem{Df} D.Dubois, H.Fargier, A.Rico, {\em Sugeno Integrals and the Commutation Problem}, (2018) In: 15th International Conference on Modeling Decisions for
Artificial Intelligence (MDAI 2018), 15 October 2018 - 18 October
2018 (Palma de Mallorca, Spain).

\bibitem{DP} D.Dubois, H.Prade, R.Sabbadin, {\em Qualitative decision theory with Sugeno integrals}, arxiv.org 1301.7372 (2014)

\bibitem{DP1} D. Dubois, J.-L. Marichal, H. Prade, M. Roubens, R. Sabbadin, {\em The use of the discrete Sugeno integral
in decision making: a survey}, Internat. J. Uncertainty, Fuzziness Knowledge-Based Systems {\bf 9} (5) (2001)
539–-561.

\bibitem{DP2} D. Dubois,  H. Prade,  {\em Fuzzy logics and the generalized modus ponens revisited.}, Cybernetics and Systems {\bf 15 } (1984) 293–-331.

\bibitem{EK} J.Eichberger, D.Kelsey, {\em Non-additive beliefs and strategic equilibria}, Games Econ Behav {\bf 30} (2000) 183--215.

\bibitem{FF} V.V.Fedorchuk, V.V.Filippov, General topology. Fundamental
constructions, Moscow,  1988.

\bibitem{FGM} T. Flaminio, L. Godo, E. Marchioni, {\em Geometrical aspects of possibility measures on finite domain MV-clans}, Soft Comput {\bf 11} (2000) 1863--1873.



\bibitem{Gil} I.Gilboa, {\em Expected utility with purely
subjective non-additive probabilities}, J. of Mathematical
Economics {\bf 16} (1987) 65--88.

\bibitem{GM} D.Glycopantis, A.Muir, {\em Nash equilibria with Knightian uncertainty; the case of capacities}, Econ. Theory {\bf 37} (2008) 147--159.

\bibitem{HM} H.Hosni, E.Marchioni, {\em Possibilistic randomisation in strategic-form games}, International Journal of Approximate Reasoning {\bf 114} (2019) 204--225.

\bibitem{PRP} E.P.Klement, R.Mesiar and E.Pap. {\em Triangular Norms}. Dordrecht: Kluwer. 2000.

\bibitem{KZ} R.Kozhan, M.Zarichnyi, {\em Nash equilibria for games in capacities}, Econ. Theory {\bf 35} (2008) 321--331.

\bibitem{L} Q.Luo, {\em KKM and Nash equilibria type theorems in topological ordered spaces}, J. Math. Anal. Appl. {\bf 264} (2001) 262--269.

\bibitem{Ma} M.Marinacci, {\em Ambiguous Games}, Games and Economic Behavior {\bf 31} (2000) 191--219.

\bibitem{NZ} O.R.Nykyforchyn, M.M.Zarichnyi, {\em Capacity functor in the category of compacta}, Mat.Sb. {\bf 199} (2008) 3--26.



\bibitem{R1} T.Radul, {\em Convexities generated by L-monads},  Applied Categorical Structures {\bf 19} (2011) 729--739.

\bibitem{R3} T.Radul, {\em Nash equilibrium for binary convexities},  Topological Methods in Nonlinear Analysis {\bf 48} (2016) 555--564.

\bibitem{R4} T.Radul, {\em Equilibrium under uncertainty with Sugeno payoff},  Fuzzy Sets and sytems {\bf 349} (2018) 64--70.

\bibitem{R5} T.Radul, {\em Games in possibility capacities with  payoff expressed by fuzzy integral},  Fuzzy Sets and systems (submitted).

\bibitem{CH} A. Rico, M. Grabisch, Ch. Labreuchea, A. Chateauneuf {\em Preference modeling on totally ordered sets by the
Sugeno integral}, Discrete Applied Mathematics {\bf 147} (2005) 113--124.

\bibitem{Sch} D.Schmeidler, {\em Subjective probability
and expected utility without additivity}, Econometrica {\bf 57} (1989) 571--587.

\bibitem{Sua}  F. Suarez, {\em Familias de integrales difusas y medidas de entropia relacionadas}, Thesis, Universidad
de Oviedo, Oviedo (1983).

\bibitem{vV}   M.van de Vel, {\em Theory of convex strutures}, North-Holland, 1993.

\bibitem{Vi} X. Vives, {\em Nash equilibrium with strategic complementarities}, J. Math. Econom. 19 (1990) 305--321.


\bibitem{WK} Zhenyuan Wang, George J.Klir   {\em Generalized measure theory}, Springer, New York, 2009.

\bibitem{We1}  S. Weber {\em Decomposable measures and integrals for archimedean t-conorms}, J. Math. Anal. Appl. {\bf 101} (1984), 114--138.

\bibitem{We2}  S. Weber {\em Two integrals and some modified versions - Critical remarks}, Fuzzy Sets and Systems {\bf 20} (1986), 97--105.

\bibitem{W}  A.Wieczorek {\em The Kakutani property and the fixed point property of topological spaces with abstract convexity,} J. Math. Anal. Appl. {\bf 168} (1992), 483--499.

\end{thebibliography}
\end{document}